\documentclass[12pt]{article}
\usepackage{amsfonts}
\usepackage{latexsym}
\usepackage{amsmath}
\usepackage{amssymb}
\usepackage{amsthm}
\usepackage{graphicx}
\usepackage{enumerate}
\usepackage{color}
\newtheorem{theorem}{Theorem}[section]
\newtheorem{lemma}[theorem]{Lemma}

\newtheorem{proposition}[theorem]{Proposition}
\newtheorem{corollary}[theorem]{Corollary}
\newtheorem{remark}[theorem]{Remark}

\newtheorem{problem}[theorem]{Problem}
\theoremstyle{definition}
\newtheorem{definition}[theorem]{Definition}

\newtheorem{thmy}{Theorem}

\newtheorem*{note*}{Note}



\begin{document}


\title{\bf On some problems concerning symmetrization operators}
\date{}
\medskip

\author {Christos Saroglou}

\maketitle

\begin{abstract}
\footnotesize In [G. Bianchi, R. J. Gardner and P. Gronchi, Symmetrization in Geometry, Adv. Math., vol. 306 (2017), 51-88], a systematic study of symmetrization operators on convex sets and their properties is conducted. In the end of their article, the authors pose several open questions. The primary goal of this manuscript is to study these questions.
\end{abstract}
\section{Introduction}\label{s-intro}

\hspace*{1.5em}In the past few years, problems of characterizing operators between convex sets by their properties have been extensively studied (see e.g. \cite{G-H-W-1}, \cite{G-H-W-2}, \cite{Mo}, \cite{S-W}, \cite{S-S}). In \cite{Bi-Ga-Gr}, the authors study symmetrization operators (see also \cite{girls} for problems concerning characterizations of symmetrization operators that are simultaneously Minkowski valuations). Our goal is to study the problems that they propose in their paper. It should be mentioned that symmetrization methods are very powerful tools for proving isoperimetric-type inequalities such as the classical Brunn-Minkowski inequality, the classical isoperimetric inequality, the Busemann centroid inequality and its recently established extensions and the Blaschke-Santal\'{o} inequality (see e.g. \cite{G2}, \cite{Bu}, \cite{LYZ}, \cite{LYZ2}, \cite{H-S}, \cite{Me-Pa}). Before describing our results, let us first fix some notation and terminology (closely following \cite{Bi-Ga-Gr}).

The origin in $\mathbb{R}^n$ is denoted by $o$ and the unit (Euclidean) ball in $\mathbb{R}^n$ is denoted by $B_2^n$. Set $G_{n,i}$ to be the family of all subspaces (i.e. the Grassmannian) of $\mathbb{R}^n$ of dimension $i\leq n$. The orthogonal projection of a set or a point $A$ in $\mathbb{R}^n$ onto a subspace $H$ is denoted by $A|H$. Set also $H^\perp$ to be the subspace which is orthogonal to $H$, where $H$ is a subspace or a vector. The support function $h_K:\mathbb{R}^n\to\mathbb{R}$ of a convex set $K$ in $\mathbb{R}^n$ is defined as follows:
$$h_K(x)=\max\{\langle x,y\rangle:y\in K\},\qquad x\in\mathbb{R}^n,$$
where $\langle\cdot,\cdot\rangle$ denotes the standard scalar product in $\mathbb{R}^n$. It is well known (see e.g. \cite{S} or \cite{G}) that $h_K$ is monotone with respect to inclusion (i.e. $h_K\leq h_L$ whenever $K\subseteq L$), additive with respect to Minkowski sum (i.e. $h_{K+L}=h_K+h_L$, where $K+L:=\{x+y:x\in K,y\in L\}$) and positively homogeneous (i.e. $h_{tK}(x)=h_K(tx)=th_K(x)$, for all $t>0$ and $x\in\mathbb{R}^n$). Moreover, if $o\in K$ and $u\in \mathbb{S}^{n-1}$, then $h_K(u)$ equals the distance of the origin to the supporting hyperplane of $K$, whose outer unit normal vector is $u$. Here, $\mathbb{S}^{n-1}=\{x\in\mathbb{R}^n:|x|=1\}$ denotes the unit sphere.

The family of all compact convex subsets of $\mathbb{R}^n$ is denoted by ${\cal K}^n$ and the family of all convex bodies (i.e. compact convex sets with non-empty interior) in $\mathbb{R}^n$ is denoted by ${\cal K}_n^n$. Let $H$ be a subspace of $\mathbb{R}^n$. We say that a set $A$ is $H$-symmetric if $A$ is symmetric with respect to $H$. In the case of a set $A$ being $\{o\}$-symmetric, we will simply write ``$A$ is $o$-symmetric". If ${\cal S}(H)$ is the family of all $H$-symmetric subsets of $\mathbb{R}^n$ and ${\cal B}$ is a family of convex sets, define ${\cal B}_H:={\cal B}\cap {\cal S}(H)$. In particular, ${\cal K}^n_{H}={\cal K}^n\cap {\cal S}(H)$ and ${\cal K}^n_{n,H}={\cal K}^n_n\cap {\cal S}(H)$.

Recall the definition of intrinsic volumes in $\mathbb{R}^n$, a re-normalization of the classical notion of quermassintegrals (see e.g. \cite{S} or \cite{G}).
A convenient way to define intrinsic volumes (without the definition of mixed volumes and quermassintegrals) is through Kubota's formula. For $j\in\{1,\dots,n\}$ and $K\in{\cal K}^n$, the $j$-th intrinsic volume $V_j(K)$ of $K$ is defined as
$$V_j(K)=\binom{n}{j}^{-1}\int_{G_{n,j}}Vol_H(K|H)d\sigma_{n,j}(H),$$where $Vol_H(\cdot)$ stands for the volume functional on $H$ and $\sigma_{n,j}$ is the Haar probability measure on the Grassmannian $G_{n,j}$. Hence $V_j(K)>0$ if and only if $\textnormal{dim}K\geq j$. It is immediate to see from the previous formula that $V_j$ is increasing on ${\cal K}^n$, strictly increasing on ${\cal K}_n^n$ and strictly increasing on ${\cal K}^n$ if $j=1$. Moreover, $V_j$ is continuous on ${\cal K}^n$ and translation invariant. Throughout this paper, a set function will be called \textit{continuous} if it is continuous with respect to the Hausdorff metric. Another remarkable property of $V_j$ is that it is independent of the dimension $n$ of the ambient space. Additionally, $V_n$ is just the $n$-dimensional volume, while $V_{n-1}$ and $V_1$ are proportional to the surface area and mean width functionals respectively. Recall that the mean width $W(K)$ of a compact convex set $K$ is defined by
$$W(K)=2\int_{\mathbb{S}^{n-1}}h_K(u)d\sigma(u),$$where $\sigma$ is the rotation invariant probability measure on $\mathbb{S}^{n-1}$.

A \textit{symmetrization} is any operator $\diamondsuit:{\cal B}\to{\cal B}_H$, where ${\cal B}$ is any class of convex sets and $H$ is a subspace of $\mathbb{R}^n$. For our purposes, ${\cal B}$ will always be ${\cal K}^n$ or ${\cal K}^n_n$. $\diamondsuit$ is called \textit{monotonic} (resp. strictly monotonic), if for any $K,L\in{\cal B}$, with $K\subseteq L$ (resp. $K\subsetneqq L$), it holds $\diamondsuit K\subseteq \diamondsuit L$ (resp. $\diamondsuit K\subsetneqq \diamondsuit L$). $\diamondsuit$ is called \textit{idempotent}, if $\diamondsuit\diamondsuit K=\diamondsuit K$, for all $K\in{\cal B}$. $\diamondsuit$ will be called \textit{invariant on $H$-symmetric sets}, if $\diamondsuit K=K$, for all $K\in{\cal B}_H$. Any set of the form $(r(B_2^n\cap H)+x)+s (B_2^n\cap H^\perp)$, where $r,s>0$ and $x\in H$, will be called an $H$-\textit{symmetric spherical cylinder}. We call $\diamondsuit$ \textit{invariant on $H$-symmetric spherical cylinders}, if $\diamondsuit B=B$, for any $H$-symmetric spherical cylinder $B$. Similarly, $\diamondsuit$ is called \textit{invariant under translations orthogonal to $H$ of $H$-symmetric sets}, if for any $K\in{\cal B}_H$, and for any $x\in H^\perp$, it holds $\diamondsuit(K+x)=K$. Let $F:{\cal B}\to \mathbb{R}$ be a set-function. Then, $\diamondsuit$ is called $F$-\textit{preserving}, if $F(\diamondsuit K)=F(K)$, for all $K\in {\cal B}$. Finally, we call a symmetrization $\diamondsuit:{\cal B}\to{\cal B}_H$ \textit{projection invariant} if $(\diamondsuit K)|H=K|H$, for all $K\in{\cal B}$.

Let us describe the two most famous examples of symmetrization. For $H\in G_{n,n-1}$ and $K\in{\cal K}^n$, the Steiner symmetral $S_H(K)$ of $K$ is the convex set obtained by translating any chord of $K$ of the form $K\cap (H^\perp+x)$, $x\in H$, so that it is $H$-symmetric and contains $x$. If we allow $H$ to be of any dimension (less than or equal to $n$), one can define the Minkowski symmetral $M_H(K)$ of $K$ as follows: $M_H(K)=(1/2)(K+K^H)$, where $K^H$ is the reflection of $K$ with respect to $H$. It is well known (see again \cite{S} or \cite{G}) that Steiner and Minkowski symmetrization are both monotonic, idempotent, invariant on $H$-symmetric spherical cylinders, invariant under translations orthogonal to $H$ of $H$-symmetric sets and projection invariant. Moreover, Steiner symmetrization is $V_n$-preserving and Minkowski symmetrization is $V_1$-preserving. Several important characterizations of Steiner and Minkowski symmetrization, by (some of) the properties mentioned above, where given in \cite{Bi-Ga-Gr}.

Since Steiner and Minkowski symmetrization possess all these nice properties described previously, it is natural to ask if there are other symmetrization operators that behave like them or if some of these properties imply the others. We are now ready to state the questions (which are of this flavor) asked in \cite{Bi-Ga-Gr} and, at the same time, describe the main results of this note.	
\begin{problem}\label{problm-1}
Let $i\in\{0,\dots,n-1\}$, $j\in\{2,\dots,n-1\}$, $H\in G_{n,i}$, ${\cal B}\in\{{\cal K}_n^n,{\cal K}^n\}$. Is there a symmetrization $\diamondsuit:{\cal B}\to{\cal B}_H$ which is monotonic, $V_j$-preserving and invariant on $H$-symmetric sets?
\end{problem}
In Section \ref{s-V_j}, we show that there is no such symmetrization, provided that $j\leq n-i$. Nevertheless, the general problem remains open, with the case $i=n-1$ being (in our opinion) the most interesting. It should be remarked that the authors in \cite{Bi-Ga-Gr} asked, in the case $i=n-1$, other variants of Problem \ref{problm-1} as well; namely what happens if $\diamondsuit$ is assumed to be invariant on $H$-symmetric spherical cylinders or projection invariant instead of invariant on $H$-symmetric sets.
\begin{problem}\label{problm-2}Let $i\in\{1,\dots,n-1\}$, $H\in G_{n,i}$, ${\cal B}\in\{{\cal K}_n^n,{\cal K}^n\}$ and $\diamondsuit:{\cal B}\to{\cal B}_H$ a strictly monotonic and idempotent symmetrization, which is invariant on $H$-symmetric spherical cylinders. Is it true that $\diamondsuit$ is projection invariant?
\end{problem}
\begin{problem}\label{problm-3}Let $i\in\{1,\dots,n-1\}$, $H\in G_{n,i}$, ${\cal B}\in\{{\cal K}_n^n,{\cal K}^n\}$ and $\diamondsuit:{\cal B}\to{\cal B}_H$ a monotonic symmetrization, which is invariant on $H$-symmetric spherical cylinders. Assume, furthermore, that there exists a strictly monotonic set-function $F:{\cal B}\to [0,\infty]$, such that $\diamondsuit$ is $F$-preserving. Is it true that $\diamondsuit$ is projection invariant?
\end{problem}
In Section \ref{s-sph.cyl.}, we show that the answer to both Problems \ref{problm-2} and \ref{problm-3} is affirmative. We remark that the cases $i=1$ and $i=n-1$ were already settled in \cite{Bi-Ga-Gr}.
\begin{problem}\label{problm-4}Let $i\in\{0,\dots,n-1\}$, $H\in G_{n,i}$, ${\cal B}\in\{{\cal K}_n^n,{\cal K}^n\}$ and $\diamondsuit:{\cal B}\to{\cal B}_H$ a monotonic and $V_1$-preserving symmetrization, which is invariant on $H$-symmetric sets. Is it true that $\diamondsuit$ is invariant under translations orthogonal to $H$ of $H$-symmetric sets?
\end{problem}
We note that if one could prove that the answer to the previous question is affirmative, then it would follow that there is no symmetrization, other than the Minkowski symmetrization, satisfying the assumptions of Problem \ref{problm-4} (see Section \ref{s-V_1} below or \cite{Bi-Ga-Gr}). We are unable to solve Problem \ref{problm-4}. We show in Section \ref{s-V_1}, however, that if ${\cal B}={\cal K}^n$ and we assume, furthermore, that $\diamondsuit$ maps line segments, contained in a translate of $H^\perp$, to line segments (see Definition \ref{def-segment-property} and Theorem \ref{thm3} below for the extension to the case ${\cal B}={\cal K}_n^n$), then the answer to this question is indeed affirmative. We should mention though that, as discussed in \cite{girls}, an assumption of the form ``segments are mapped to segments" is a rather strong one.
\section{The natural extension}\label{s-n.e.}
\hspace*{1.5em}We will make use of the following definition in Sections \ref{s-V_j} and \ref{s-V_1}.
\begin{definition}\label{def-natural-extension}
Let $n\in\mathbb{N}$, $i\in\{0,\dots ,n-1\}$, $H\in G _{n,i}$ and $\diamondsuit:{\cal K}_n^n\to {\cal K}^n_{n,H}$ an $H$-symmetrization. The {\it natural extension} $\overline{\diamondsuit}:{\cal K}^n\to {\cal K}_H^n$ of $\diamondsuit$ is defined as follows:$$\overline{\diamondsuit}K=\bigcap _{m=1}^{\infty}\diamondsuit\left(K+\frac{1}{m}B_2^n\right),\qquad K\in{\cal K}^n.$$
\end{definition}
It should be remarked that the restriction of $\overline{\diamondsuit}$ onto ${\cal K}_n^n$ does not need to coincide with $\diamondsuit$ (even under the assumption of strict monotonicity) unless certain additional  properties are assumed for $\diamondsuit$. For instance, for $K\in{\cal K}_n^n$, set $B_K$ to be the $o$-symmetric Euclidean ball of volume $V_n(K)$ and define
$$\diamondsuit K:=\begin{cases}
B_K, & V_n(K)\leq 1\\
2B_K, & V_n(K)>1.
\end{cases}$$
Then, for any $i\in\{0,\dots,n-1\}$ and any $H\in G_{n,i}$, $\diamondsuit:{\cal K}_n^n\to {\cal K}_{n,H}^n$ is a strictly monotonic symmetrization, with $\overline{\diamondsuit}K=2\diamondsuit K\neq \diamondsuit K$, if $V_n(K)=1$.

The next lemma summarizes some basic information concerning the natural extension operator.
\begin{lemma}\label{lemma-natural-extension}
Let $i\in\{0,\dots ,n-1\}$, $H\in G _{n,i}$, $\diamondsuit:{\cal K}_n^n\to {\cal K}^n_{n,H}$ an $H$-symmetrization and $\overline{\diamondsuit}:{\cal K}^n\to {\cal K}_H^n$ be its natural extension.
Assume that $\diamondsuit$ is monotonic.
Then, the following hold true:
\begin{enumerate}[i)]
\item $\overline{\diamondsuit}K\supseteq \diamondsuit K$, for all $K\in{\cal K}^n_n$. Moreover, if there exists a strictly monotonic and continuous set function $F:{\cal K}^n_n\to{\cal K}^n_{n,H}$ such that $\diamondsuit$ is $F$-preserving, then $\overline{\diamondsuit}K= \diamondsuit K$, for all $K\in{\cal K}^n_n$.
\item If $\diamondsuit$ is $V_j$-preserving, for some $j\in\{1,\dots ,n\}$, then $\overline{\diamondsuit}$ is $V_j$-preserving.
\item $\diamondsuit$ is projection invariant if and only if $\overline{\diamondsuit}$ is projection invariant.
\item If $\diamondsuit$ is invariant on $H$-symmetric sets, then $\overline{\diamondsuit}$ is invariant on $H$-symmetric sets.
\item If $\diamondsuit$ is invariant on $H$-symmetric spherical cylinders, then $\overline{\diamondsuit}$ is invariant on $H$-symmetric spherical cylinders.
\item If $\diamondsuit$ is invariant under translations orthogonal to $H$ of $H$-symmetric sets, then $\overline{\diamondsuit}$ is invariant under translations orthogonal to $H$ of $H$-symmetric sets. Moreover, if $\overline{\diamondsuit}$ is invariant under translations orthogonal to $H$ of $H$-symmetric sets, then $\diamondsuit (K+x)\supseteq K$, for all $H$-symmetric sets $K$ and for all $x\in H^{\perp}$. If, in addition, $\diamondsuit$ is $V_j$-preserving, for some $j\in\{1,\dots, n\}$, then $\diamondsuit$ is invariant under translations orthogonal to $H$ of $H$-symmetric sets.
\end{enumerate}

\end{lemma}
\begin{proof}
(i) The fact that $\overline{\diamondsuit}K$ contains $\diamondsuit K$, for $K\in{\cal K}^n_n$ is trivial. Assume that there exists a continuous set function $F:{\cal K}^n_n\to{\cal K}^n_{n,H}$ such that $\diamondsuit$ is $F$-preserving. Then, by the monotonicity of $F$, $F(\overline{\diamondsuit}K)\geq F(\diamondsuit K)=F(K)$. Moreover, the sequence $\left\{K+\frac{1}{m}B_2^n\right\}$ converges to $K$, with respect to the Hausdorff metric, thus $$F(\overline{\diamondsuit}K)\leq F\left(K+\frac{1}{m}B_2^n\right)\to F(K),\qquad \textnormal{as }m\to \infty,$$
hence $F(\overline{\diamondsuit}K)=F(\diamondsuit K)$. Since $F$ is strictly monotonic, it follows that $\overline{\diamondsuit}K=\diamondsuit K$.\\
(ii) As noted in the Introduction, $V_j$ is continuous in ${\cal K}^n$. Fix $K\in {\cal K}^n$. Notice that the sequence $\left\{K+\frac{1}{m}B_2^n\right\}_{m=1}^{\infty}$ is decreasing, therefore by the monotonicity of $\diamondsuit$, $\left\{\diamondsuit\left(K+\frac{1}{m}B_2^n\right)\right\}_{m=1}^{\infty}$ is also decreasing and consequently $\diamondsuit\left(K+\frac{1}{m}B_2^n\right)\to \overline{\diamondsuit}K$, as $m\to \infty$. Thus, as before, we have:
$$V_j(\overline{\diamondsuit}K)=\lim_{m\to\infty}V_j\left(K+\frac{1}{m}B_2^n\right)=V_j(K).$$
(iii) First assume that $\diamondsuit $ is projection invariant. Let $K\in {\cal K}^n$. Then, for every $m\in \mathbb{N}$,
$$\left(\overline{\diamondsuit}K\right)|H \subseteq\left(\diamondsuit\left(K+\frac{1}{m}B_2^n\right)\right)|H=\left(K+\frac{1}{m}B_2^n\right)|H=(K|H)+\frac{1}{m}(B_2^n|H)\xrightarrow{m\to\infty} K|H.$$
On the other hand, by $(i)$, we have $\overline{\diamondsuit}K\supseteq \diamondsuit K$, thus $\left(\overline{\diamondsuit}K\right)|H\supseteq \left(\diamondsuit K\right)|H$, which shows that $\left(\overline{\diamondsuit}K\right)|H= \left(\diamondsuit K\right)|H$. Conversely, assume that $\overline{\diamondsuit}$ is projection invariant. Let $K\in{\cal K}_n^n$. We already know that $K|H=\left(\overline{\diamondsuit}K\right)|H\supseteq \left(\diamondsuit K\right)|H$. To prove the inverse inclusion, fix $y\in \textnormal{int}K$. Since $(\diamondsuit K)|H$ is assumed to be a closed set, it suffices to show that $y|H\in( \diamondsuit  K)|H$. By assumption, $\overline{\diamondsuit}\{y\}|H=\{y\}|H$, thus $y|H\in \diamondsuit\left(\{y\}+\frac{1}{m}B_2^n\right)=\diamondsuit\left(\frac{1}{m}B_2^n+y\right)$, for all $m\in \mathbb{N}$. On the other hand, since $y\in \textnormal{int}K$, there exists $m_0\in\mathbb{N}$, such that $\frac{1}{m}B_2^n+y\subseteq K$, for all $m\geq m_0$. Thus,
$$(\diamondsuit K)|H\supseteq\bigcap_{m=m_0}^{\infty}\left(\diamondsuit\left(\frac{1}{m}B_2^n+y\right)|H\right)=\left(\bigcap_{m=1}^{\infty}\left(\frac{1}{m}B_2^n+y\right)\right)|H=\{y|H\},$$
as asserted.\\
(iv) Let $K\in {\cal K}^n$ be an $H$-symmetric set. Then, $K+\frac{1}{m}B_2^n$ is also $H$-symmetric, thus
$$\overline{\diamondsuit}K=\bigcap _{m=1}^{\infty}\diamondsuit\left(K+\frac{1}{m}B_2^n\right)=\bigcap _{m=1}^{\infty}\left(K+\frac{1}{m}B_2^n\right)=K.$$
(v) Let $m\in\mathbb{N}$, $s,t>0$ and $x\in H$. Then, by $(i)$ and the assumption, we have
\begin{eqnarray*}\overline{\diamondsuit}\left(\big(s(B_2^n\cap H)+t(B_2^n\cap H^{\perp})\big)+x\right)&\supseteq& \diamondsuit\left(\big(s(B_2^n\cap H)+t(B_2^n\cap H^{\perp})\big)+x\right)\\
&=&\left(\big(s(B_2^n\cap H)+t(B_2^n\cap H^{\perp})\big)+x\right).
\end{eqnarray*}
Moreover,
$$\left(\big(s(B_2^n\cap H)+t(B_2^n\cap H^{\perp})\big)+x\right)+\frac{1}{m}B_2^n\subseteq \left(\frac{ms+1}{m}(B_2^n\cap H)+\frac{mt+1}{m}(B_2^n\cap H^{\perp})\right)+x,$$
thus\\$\overline{\diamondsuit}\left(\big(s(B_2^n\cap H)+t(B_2^n\cap H^{\perp})\big)+x\right)$
\begin{eqnarray*}
&\subseteq& \bigcap_{m=1}^{\infty}\diamondsuit\left(\left(\frac{ms+1}{m}(B_2^n\cap H)+\frac{mt+1}{m}(B_2^n\cap H^{\perp})\right)+x\right)\\
&=&\bigcap_{m=1}^{\infty}\left(\left(\frac{ms+1}{m}(B_2^n\cap H)+\frac{mt+1}{m}(B_2^n\cap H^{\perp})\right)+x\right)\\
&\xrightarrow{m\to\infty}&\big(s(B_2^n\cap H)+t(B_2^n\cap H^{\perp})\big)+x.
\end{eqnarray*}
(vi) Fix an $H$-symmetric set $K\in {\cal K}_n^n$ and $x\in H^{\perp}$. If $\diamondsuit$ is invariant under translations orthogonal to $H$ of $H$-symmetric sets, we have:
$$\overline{\diamondsuit} (K+x)=\bigcap_{m=1}^{\infty}\diamondsuit \left(K+\frac{1}{m}B_2^n+x\right)=\bigcap_{m=1}^{\infty} \left(K+\frac{1}{m}B_2^n\right)=K,$$thus $\overline{\diamondsuit}$ is invariant under translations orthogonal to $H$ of $H$-symmetric sets. Conversely, if  $\overline{\diamondsuit}$ is invariant under translations orthogonal to $H$ of $H$-symmetric sets, then by $(i)$, we have: $K=\overline{\diamondsuit}(K+x)\supseteq \diamondsuit(K+x)$. If, in addition, $\diamondsuit$ is $V_j$-preserving, then $V_j(\diamondsuit(K+x))=V_j(K+x)=V_j(K)$, which shows that $\diamondsuit(K+x)=K$. Therefore, $\diamondsuit$ is invariant under translations orthogonal to $H$ of $H$-symmetric sets.
\end{proof}
\section{Symmetrization operators invariant on $H$-symmetric spherical cylinders}\label{s-sph.cyl.}
\hspace*{1.5em}This section is devoted to the presentation of the solution of Problems \ref{problm-2} and \ref{problm-3} .
\begin{theorem}\label{thm1}
Let $i\in\{1,\dots, n-1\}$, $H\in G _{n,i}$ and $\diamondsuit:{\cal B}\to {\cal B}_H$ a symmetrization, where ${\cal B}={\cal K}^n $or ${\cal K}_n^n$. If $\diamondsuit$ is strictly monotonic, idempotent and invariant on $H$-symmetric spherical cylinders, then $\diamondsuit$ is projection invariant.
\end{theorem}

\begin{theorem}\label{thm2}
Let $i\in\{1,\dots, n-1\}$, $H\in G _{n,i}$ and $\diamondsuit:{\cal B}\to {\cal B}_H$ a symmetrization, where ${\cal B}={\cal K}^n $or ${\cal K}_n^n$.
Assume that there exists a strictly increasing set function $F:{\cal B}\to [0,\infty)$ such that $\diamondsuit$ is $F$-preserving. If $\diamondsuit$ is monotonic and invariant on $H$-symmetric spherical cylinders, then $\diamondsuit$ is projection invariant.
\end{theorem}
To avoid distinguishing two cases: ${\cal B}={\cal K}_n^n$ or ${\cal B}={\cal K}^n$, we will consider a slightly more general map, namely $\diamondsuit:{\cal K}_n^n\to {\cal K}^n_H$. The following lemma shows that it suffices to prove Theorems \ref{thm1} and \ref{thm2} with ${\cal B}={\cal K}_n^n$ and with ${\cal K}_H^n$ in the place of ${\cal B}_H$.
\begin{lemma}\label{lemma-reduction-to-1case}
Let $H$ be a proper subspace of $\mathbb{R}^n$ and $\diamondsuit:{\cal K}^n\to {\cal K}^n_H$ a monotonic map. If the restriction $\diamondsuit|_{{\cal K}_n^n}$ of $\diamondsuit$ onto the family ${\cal K}_n^n$ is projection invariant, then $\diamondsuit$ is projection invariant.
\end{lemma}
\begin{proof}
Let $P\in {\cal K}^n\setminus {\cal K}^n_n$. Clearly, there exists $K\in{\cal K}^n$, such that $K\supseteq P$ and $K|H=P|H$. By the monotonicity of $\diamondsuit$, we get
$$(\diamondsuit P)|H\subseteq (\diamondsuit K)|H=K|H=P|H.$$
To prove the other inclusion, as in Lemma \ref{lemma-natural-extension} $(iii)$, we need to show that $(\diamondsuit \{y\})|H=\{y\}|H$, for any $y\in\mathbb{R}^n$. Since $$(\diamondsuit \{y\})|H\subseteq \diamondsuit\left(\frac{1}{m}B_2^n+y\right)|H=\left(\frac{1}{m}B_2^n+y\right)|H=\frac{1}{m}(B_2^n|H)+(y|H),$$for all $m\in\mathbb{N}$, it follows that
$$(\diamondsuit \{y\})|H\subseteq \bigcap_{m=1}^{\infty}\left(\frac{1}{m}(B_2^n|H)+(y|H)\right)=\{y|H\}.$$But since it is assumed that $\diamondsuit\{y\}\neq \emptyset$, it follows that $(\diamondsuit \{y\})|H=\{y|H\}$, as required.
\end{proof}
We will need some additional lemmas.
\begin{lemma}\label{lemma-projection-of-symmetrization-contains-projection}
Let $H$ be a proper subspace of $\mathbb{R}^n$ and $\diamondsuit:{\cal K}^n\to {\cal K}^n_H$ a monotonic map. If $\diamondsuit$ is invariant on $H$-symmetric spherical cylinders, then $(\diamondsuit K)|H\supseteq K|H$.
\end{lemma}
\begin{proof}
First let us fix some notation. For $\delta>0$, $x\in H$, set $$C(\delta,x):=\delta (B_2^n\cap H)+\delta (B_2^n\cap H^{\perp})+x. $$Thus, $C(\delta,x)$ is always an $H$-symmetric spherical cylinder. Let $K\in {\cal K}^n_n$ and $x\in \textnormal{relint}(K|H)$. Then, there exist $y\in H^{\perp}$ and $\delta_0>0$, such that $C(\delta,x)+y\subseteq K$, for all $0<\delta<\delta_0$. On the other hand, there exists a large enough $t>0$, so that $\delta (B_2^n\cap H)+t(B_2^n\cap H^{\perp})+x\supseteq C(\delta,x)+y$. Thus, by the monotonicity of $\diamondsuit$ and the fact that $\diamondsuit$ is invariant on $H$-symmetric spherical cylinders, we get:
\begin{eqnarray*}
 \left(\diamondsuit\left(C(\delta,x)+y\right)\right)|H&\subseteq& \diamondsuit\left(\left(\delta (B_2^n\cap H)+t(B_2^n\cap H^{\perp})+x\right)\right)|H\\&=&\left(\delta (B_2^n\cap H)+t(B_2^n\cap H^{\perp})+x\right)|H\\
 &=&\delta(B_2^n\cap H)+x.
 \end{eqnarray*}
Since $\diamondsuit\left(C(\delta,x)+y\right)$ is non empty and since $\diamondsuit K\supseteq \diamondsuit\left(C(\delta,x)+y\right)$, it follows that $\left((\diamondsuit K)|H\right)\cap \left(\delta(B_2^n\cap H)+x\right)\neq \emptyset$. Since $x$ is an arbitrary point in
$\textnormal{relint}(K|H)$ and $\delta$ can be as small as we want, it follows that $(\diamondsuit K)|H\supseteq \textnormal{relint}(K|H)$. Our claim follows by the fact that $(\diamondsuit K)|H$ is a closed set.
\end{proof}
\begin{definition}\label{def-sets-of-special-form}
Any set of the form $L+s(B_2^n\cap H^{\perp})$ will be called {\it set of special form}, where $s>0$ and $L$ is any convex body in $H$.
\end{definition}
\begin{lemma}\label{lemma-sets-of-special-form-1}
Let $H$ be a proper subspace of $\mathbb{R}^n$ and $\diamondsuit:{\cal K}_n^n\to {\cal K}_H^n$ be a monotonic map, which is invariant on $H$-symmetric spherical cylinders. Then, $T\subseteq\diamondsuit T$, for any set $T$ of special form.
\end{lemma}
\begin{proof}
Let $L$ be a convex body in $H$ and $s>0$. Set $T:=L+s(B_2^n\cap H^{\perp})$. We need to show that $T\subseteq\diamondsuit T$. Let $x\in\textnormal{relint}(T|H)=\textnormal{relint}L$. Then, there exists $\delta>0$, such that $\delta(B_2^n\cap H)+x\subseteq T|H$. Set
$C:=\delta(B_2^n\cap H)+s(B_2^n\cap H^{\perp})+x$. Then, $C$ is an $H$-symmetric spherical cylinder contained in $T$. Thus, $C=\diamondsuit C\subseteq \diamondsuit T$. In particular, $s(B_2^n\cap H^{\perp})+x\subseteq \diamondsuit T$ and since $x$ is an arbitrary point of $\textnormal{relint}L=\textnormal{relint}(T|H)$, it follows that $\textnormal{int}T\subseteq \diamondsuit T$. Since $\diamondsuit T$ is a closed set, the assertion follows.
\end{proof}
\begin{lemma}\label{lemma-sets-of-special-form-2}
Let $H$ be a proper subspace of $\mathbb{R}^n$ and $\diamondsuit:{\cal K}_n^n\to {\cal K}_H^n$ be a monotonic map. If $\diamondsuit T\subseteq T$, for all sets $T$ of special form, then $\diamondsuit$ is projection invariant.
\end{lemma}
\begin{proof}
Fix a convex body $K$ in $\mathbb{R}^n$. By Lemma \ref{lemma-projection-of-symmetrization-contains-projection}, it suffices to show that $(\diamondsuit K)|H\subseteq K|H$. It is clear that there exists a large enough $s>0$, so that $K|H^{\perp}$ is contained in $s(B_2^n\cap H^{\perp})$. Then, $K\subseteq T:=(K|H)+s(B_2^n\cap H^{\perp})$. Notice that $K|H=T|H$ and that $T$ is a set of special form. By our assumption and the monotonicity of $\diamondsuit$, we have:
$$(\diamondsuit K)|H\subseteq (\diamondsuit T)|H\subseteq T|H=K|H,$$ as claimed.
\end{proof}
We are now ready to prove the main results of this section. The proofs of Theorems \ref{thm1} and \ref{thm2} are similar and will follow from the lemmas we have established.\\
\\
Proof of Theorem \ref{thm1}.\\
By Lemma \ref{lemma-reduction-to-1case}, it suffices to prove that if $\diamondsuit:{\cal K}_n^n\to{\cal K}^n_H$ is a strictly monotonic, invariant on $H$-symmetric spherical cylinders and idempotent map, then $\diamondsuit$ is projection invariant. Actually, by Lemma \ref{lemma-sets-of-special-form-2}, it suffices to show that $\diamondsuit T=T$, for any set $T$ of special form. By Lemma, \ref{lemma-sets-of-special-form-1}, we have $\diamondsuit T\supseteq T$. Assume that $\diamondsuit T\neq T$. Then, by the strict monotonicity of $\diamondsuit$, it follows that $\diamondsuit \diamondsuit T$ strictly contains $\diamondsuit T$. However, $\diamondsuit$ is idempotent, thus $\diamondsuit \diamondsuit T=\diamondsuit T$. This is a contradiction, thus $\diamondsuit T=T$ and the proof
is complete. $\square$\\
\\
Proof of Theorem \ref{thm2}.\\
As before, Lemmas \ref{lemma-reduction-to-1case} and \ref{lemma-sets-of-special-form-2} show that it suffices to prove that if $\diamondsuit:{\cal K}_n^n\to{\cal K}^n_H$ is a monotonic, invariant on $H$-symmetric spherical cylinders and $F$-preserving map, where $F:{\cal K}_n^n\to {\cal K}^n_H$ is a strictly increasing set function, then $\diamondsuit T=T$, for any set $T$ of special form. Lemma \ref{lemma-sets-of-special-form-1} shows that $\diamondsuit T\supseteq T$. If $\diamondsuit T\neq T$, then by the strict monotonicity of $F$, we would have $F(\diamondsuit T)>F(T)$. This contradicts the fact that $\diamondsuit$ is $F$-preserving, thus $\diamondsuit T=T$, for all sets $T$ of special form. $\square$
\begin{remark}\label{rem-s3}
It follows by the proofs of Theorems \ref{thm1} and \ref{thm2} (a fact that was also mentioned earlier) that slightly more general versions of these two theorems (in the case that ${\cal B}={\cal K}^n_n$) are valid. Let $i\in\{1,\dots, n-1\}$, $H\in G _{n,i}$ and $\diamondsuit:{\cal K}_n^n\to {\cal K}^n_H$ be a map which is strictly monotonic and invariant on $H$-symmetric cylinders. If $\diamondsuit$ is idempotent or $F$-preserving, for some strictly increasing set-function $F:{\cal K}_n^n\to [0,\infty)$, then $\diamondsuit$ is projection invariant.
\end{remark}
\section{$V_j$-preserving symmetrization operators, $2\leq j\leq n-1$}\label{s-V_j}
\hspace*{1.5em}The main goal of this section is to solve Problem \ref{problm-1} in the case $2\leq j\leq n-i$.
\begin{theorem}\label{thm-v_j}
Let $i\in\{0,\dots,n-1\}$, $j\in\{2,\dots,n\}$, such that $j\leq n-i$ and $H\in G_{n,i}$. Set ${\cal B}={\cal K}^n $or ${\cal K}_n^n$. There does not exist a symmetrization $\diamondsuit :{\cal B}\to{\cal B}_H$ which is monotonic, invariant on $H$-symmetric sets and $V_j$-preserving.	
\end{theorem}
Before proving Theorem \ref{thm-v_j}, we will need some geometric statements.
\begin{lemma}\label{lemma-subspaces}
Let $i\in\{0,\dots,n-1\}$, $H\in G_{n,i}$ and $\diamondsuit :{\cal K}^n\to{\cal K}^n_H$ a symmetrization which is monotonic and invariant on $H$-symmetric sets. Then, for every affine subspace $G$ of $\mathbb{R}^n$ that either contains a translate of $H^\perp$ or is a linear subspace of $H^\perp$, it holds $\diamondsuit K\subseteq G$, for all $K\in{\cal K}^n$, $K\subseteq G$.
\end{lemma}
\begin{proof}
Notice that in both cases, $G$ is $H$-symmetric. Thus, if $r>0$ and $x\in G\cap H$, the ball $B_r(x):=G\cap (B_2^n+x)$ is $H$-symmetric. Let $K\in{\cal K}^n$, such that $K\subseteq G$. Choose $r>0$ to be so large that $B_r(x)\supseteq K$. Then, $\diamondsuit K\subseteq \diamondsuit B_r(x)=B_r(x)\subseteq G$, as claimed.	
\end{proof}
\begin{proposition}\label{prop-v_j-1}
Let $i\in\{0,\dots,n-1\}$, $j\in\{1,\dots,n\}$, $H\in G_{n,i}$ and $\diamondsuit :{\cal K}^n\to{\cal K}^n_H$ a symmetrization which is monotonic, $V_j$-preserving and invariant on $H$-symmetric sets. Then, $\diamondsuit K\in {\cal K}_{n,H}^n$, for all $K\in {\cal K}_{n}^n$.
\end{proposition}
\begin{proof}
Let $K\in {\cal K}_{n}^n$. We need to show that $\diamondsuit K$ is $n$-dimensional. First assume that $i=0$. It follows immediately from Lemma \ref{lemma-subspaces} that for any $j$-dimensional subspace $S$ of $\mathbb{R}^n$, it holds $\diamondsuit(K\cap S)\subseteq S$. If it happens that $V_j(K\cap S)>0$, then $V_j(\diamondsuit(K\cap S))>0$. We conclude that if $\textnormal{span}(K\cap S)=S$, i.e. $K\cap S$ is $j$-dimensional, then $\diamondsuit(K\cap S)$ is also $j$-dimensional, so (since clearly $\diamondsuit(K\cap S)\subseteq (\diamondsuit K)\cap S$) $(\diamondsuit K)\cap S$ is $j$-dimensional, hence $\textnormal{span}((\diamondsuit K)\cap S)=S$. First assume that $j=1$. Since $K$ is $n$-dimensional, there exist points $u_1,\dots,u_n\in K$, whose position vectors are linearly independent. Then, $\textnormal{span}(\diamondsuit K)\supseteq \textnormal{span}((\diamondsuit K)\cap (\mathbb{R}u_i))=\mathbb{R}u_i$, $i=1,\dots,n$, thus $\diamondsuit K$ is $n$-dimensional. Therefore, we may assume that $j\geq 2$. Set 
$y$ to be an interior point of $K$ and ${\cal S}$ to be the set of all $j$-dimensional subspaces of $\mathbb{R}^n$ that contain $y$. Then, clearly, the union of all subspaces from ${\cal S}$ is the whole $\mathbb{R}^n$ and for each $S\in{\cal S}$, it holds $\textnormal{span}(K\cap S)=S$. Thus, $\textnormal{span}(\diamondsuit K)\supseteq \bigcup_{S\in{\cal S}}\textnormal{span}((\diamondsuit K)\cap S)=\bigcup_{S\in{\cal S}}S=\mathbb{R}^n$, which again shows that $\diamondsuit K$ is $n$-dimensional.
	
In the general case, we will prove our claim by induction in $n$. The case $n=1$ is trivial. We may assume that $i>0$, otherwise we are done. Assume that the assertion holds for the positive integer $n-1$. Notice that we may assume that $n-1\geq j$, otherwise the result is again trivial. In the inductive step, assume that $\diamondsuit K$ is not full dimensional. That is, there exists a subspace $G$ of dimension $n-1$, such that $\diamondsuit  K\subseteq G$. Notice that the restriction $V_j|_{{\cal K}_n^n}$ of $V_j$ onto the family ${\cal K}_n^n$ is strictly monotonic. Thus, it follows by Remark \ref{rem-s3}, that the map $\diamondsuit|_{{\cal K}_n^n}:{\cal K}_n^n\to {\cal K}^n_H$ is projection invariant. In particular, since $\diamondsuit K$ is $H$-symmetric, we have $(\diamondsuit K)\cap H=(\diamondsuit K)|H=K|H$, which is full dimensional in $H$. Hence, $H\subseteq G$. Let $F$ be any subspace of $\mathbb{R}^n$, of dimension $n-1$, that contains $H^\perp$. 
Notice that $F$ cannot contain $H$, thus $F\neq G$. Then, there exists $x\in F^\perp\subseteq H$, such that $\textnormal{dim}(K\cap (F+x))=n-1$ (otherwise it would follow by Fubini's Theorem that $V_n(K)=0$). As before, it follows from Lemma \ref{lemma-subspaces} that if $L\in{\cal K}^n$ and $L\subseteq F+x$, then $\diamondsuit L\subseteq F+x$. This shows that the restriction $\diamondsuit'$ of $\diamondsuit$ onto the family of compact convex subsets of $F+x$ is a symmetrization on $F+x$. Since $\diamondsuit'$ is trivially monotonic, invariant on $H\cap (F+x)$-symmetric sets and $V_j$-preserving, it follows by the inductive hypothesis that $\textnormal{dim}(K\cap (F+x))=n-1$, hence $K\not\subseteq G$, which is a contradiction, proving our claim.
\end{proof}
\begin{proposition}\label{prop-v_j-2}
Let $i\in\{0,\dots,n-1\}$, $j\in\{1,\dots,n\}$, $H\in G_{n,i}$ and $\diamondsuit :{\cal K}^n_n\to{\cal K}^n_{n,H}$ a symmetrization which is monotonic, $V_j$-preserving and invariant on $H$-symmetric sets. Let $\varepsilon\subseteq H^\perp$ be a straight line containing the origin and $I\subseteq \varepsilon$ be a line segment. Then, $\overline{\diamondsuit}I$ is the $o$-symmetric line segment contained in $\varepsilon$ which has the same length as $I$, where $\overline{\diamondsuit}$ is the natural extension of $\diamondsuit$.
\end{proposition}
\begin{proof}
Set $I'$ to be the $o$-symmetric line segment contained in $\varepsilon$ which has the same length as $I$. By Lemma \ref{lemma-subspaces}, $\overline{\diamondsuit}I$ is contained in $\varepsilon$, so $\overline{\diamondsuit}I$ is an $o$-symmetric line segment (possibly degenerate). By Lemma \ref{lemma-natural-extension} (i), (ii), (v), we have $\overline{\diamondsuit}|_{{\cal K}_n^n}=\diamondsuit$, $\overline{\diamondsuit}$ is invariant on $H$-symmetric sets and $\overline{\diamondsuit}$ is $V_j$-preserving. Let $G$ be a subspace of $\varepsilon^\perp$ of dimension $j-1$ that either contains $H$ or is contained in $H$ (depending on whether $j-1\geq i$ or $j-1<i$). Then, in any case, it is clear that $G$ is $H$-symmetric, thus the ball $B:=B_2^n\cap G\cap H$ is $H$-symmetric. For $\delta>0$, define the $j$-dimensional cylinder $K_\delta:=I+\delta B$. Then,
\begin{equation}\label{eq-cyl-1}
V_j(K_\delta)=V_{j-1}(\delta B)V_1(I).
\end{equation}
Moreover, as in the proof of Proposition \ref{prop-v_j-1}, since $V_j$ is strictly increasing on ${\cal K}_n^n$, it follows by Theorem \ref{thm2} together with Lemma \ref{lemma-natural-extension} that $\diamondsuit$ and $\overline{\diamondsuit}$ are projection invariant. Hence,$$\overline{\diamondsuit}K_\delta\subseteq \left(\overline{\diamondsuit}K_\delta\right)|H+\left(\overline{\diamondsuit}K_\delta\right)|\varepsilon=\delta B+\left(\overline{\diamondsuit}K_\delta\right)|\varepsilon,$$Consequently,
\begin{equation}\label{eq-cyl-2}
V_j\left(\overline{\diamondsuit}K_\delta\right)\leq V_{j-1}(\delta B)V_1\left(\left(\overline{\diamondsuit}K_\delta\right)|\varepsilon\right) .
\end{equation}
Combining (\ref{eq-cyl-1}), (\ref{eq-cyl-2}) with the fact that $\overline{\diamondsuit}$ is $V_j$-preserving, we get $V_1\left(\left(\overline{\diamondsuit}K_\delta\right)|\varepsilon\right)\geq V_1(I)$, for all $\delta>0$.
Since $\left(\overline{\diamondsuit}K_\delta\right)|\varepsilon$ is $H$-symmetric (thus $o$-symmetric) and is contained in $\varepsilon$, we conclude that $\left(\overline{\diamondsuit}K_\delta\right)|\varepsilon\supseteq I'$, for all $\delta>0$. But by the definition of $\overline{\diamondsuit}$, we get
$$\overline{\diamondsuit}I=\left(\overline{\diamondsuit}I\right)|\varepsilon=\left(\bigcap_{m=1}^\infty\diamondsuit\left(I+\frac{1}{m}B_2^n\right)\right)|\varepsilon\supseteq \left(\bigcap_{m=1}^\infty \overline{\diamondsuit }K_{1/m}\right)|\varepsilon\supseteq I'.$$
To show the reverse inclusion, set $x$ to be one of the end-points of $I$ and consider the $j$-dimensional cone
$L:=\textnormal{conv}\left((B+x)\cup I\right)$. Then,
\begin{equation}\label{eq-cone-1}
V_j(L)=\frac{1}{j}V_{j-1}(B)V_1(I).
\end{equation}
Since, as mentioned previously, $\overline{\diamondsuit}$ is projection invariant, it follows that $\left(\overline{\diamondsuit}L\right)|H=L|H=B$.
Observe that the Steiner-symmetral $S_{\varepsilon^\perp}\left(\overline{\diamondsuit}L\right)$ of $\overline{\diamondsuit}L$, with respect to the hyperplane $\varepsilon^\perp$, contains $\overline{\diamondsuit}I$ and also contains $B$, it contains the double cone $\textnormal{conv}\left(\overline{\diamondsuit}I\cup B\right)$. Therefore,
\begin{equation}\label{eq-cones-2}
V_j\left(\overline{\diamondsuit}L\right)=V_j\left(S_{\varepsilon^\perp}\left(\overline{\diamondsuit}L\right)\right)\geq \frac{1}{j}V_{j-1}(B)V_1\left(\overline{\diamondsuit}I\right).
\end{equation}
Using again the fact that $\overline{\diamondsuit}$ is $V_j$-preserving, together with (\ref{eq-cone-1}) and (\ref{eq-cones-2}), we see that
$$V_1\left(\overline{\diamondsuit}I\right)\leq V_1(I)=V_1(I').$$
Since both segments $\overline{\diamondsuit}I$ and $I'$ are contained in $\varepsilon$ and they are both $o$-symmetric, it follows that $\overline{\diamondsuit}I\subseteq I'$, as required.
\end{proof}
\begin{lemma}\label{lemma-triangle}
Let $\diamondsuit:{\cal K}^n\to{\cal K}^n$ be a monotonic map with the property that for any line segment that contains $o$, $\diamondsuit I$ is the $o$-symmetric line segment which is parallel to $I$ and has the same length as $I$. Let $F$ be a 2-dimensional subspace of $\mathbb{R}^n$ and $T\subseteq F$ be an equilateral triangle whose barycenter is at the origin. Then, $V_2((\diamondsuit T)\cap F)\geq V_2(T)$.	
\end{lemma}
\begin{proof}
Let $v_1,v_2,v_3$ be the vertices of $T$. Set $a:=|v_1-v_2|=|v_2-v_3|=|v_1-v_3|$. Then, the distance of $v_i$ from the side of $T$ which is opposite to $v_i$ equals $a\sqrt{3}/2$, $i=1,2,3$. It follows that $V_2(T)=\sqrt{3}a^2/8$. Furthermore, for each $i\in\{1,2,3\}$, $T$ contains a line segment parallel to $[o,v_i]$, of length $\sqrt{3}a/2$. Thus, for $i\in\{1,2,3\}$, the line segment $(\sqrt{3}a/(4|v_i|)[-v_i,v_i]$ is contained in $(\diamondsuit T)\cap F$. In other words, since $(\diamondsuit T)\cap F$ is convex, it contains the regular hexagon $H$ with vertices $\pm \sqrt{3}av_i/(4|v_i|)$, $i=1,2,3$. Let $S$ be the equilateral triangle with vertices $\sqrt{3}av_1/(4|v_1|)$, $\sqrt{3}av_2/(4|v_2|)$ and $o$. Since, clearly,
$$V_2(H)=6V_2(S)=\frac{6\sqrt{3}}{8}\left(\frac{\sqrt{3}a}{4}\right)^2=\frac{18}{16}\frac{\sqrt{3}a^2}{8}>\frac{\sqrt{3}a^2}{8}=V_2(T),$$
we conclude that $V_2((\diamondsuit T)\cap F)\geq V_2(H)>V_2(T)$, as claimed.
\end{proof}
\textit{}\\
Proof of Theorem \ref{thm-v_j}.\\
Assume that there exists such a $\diamondsuit$. If ${\cal B}={\cal K}^n$, we know from Proposition \ref{prop-v_j-1} that the restriction of $\diamondsuit$ onto ${\cal K}_n^n$ is a symmetrization $:{\cal K}_n^n\to {\cal K}_{n,H}^n$ (which is also monotonic, invariant on $H$-symmetric sets and $V_j$-preserving). Therefore, we may assume that ${\cal B}={\cal K}_n^n$. Denote by $\overline{\diamondsuit}$ the natural extension of $\diamondsuit$. We know by Lemma \ref{lemma-natural-extension} that $\overline{\diamondsuit}$ is also monotonic, invariant on $H$-symmetric sets and $V_j$-preserving. Moreover, it follows by Proposition \ref{prop-v_j-2} that for every line segment $I$, that contains the origin, $\overline{\diamondsuit}I$ is the $o$-symmetric line segment, which is parallel to $I$ and has the same length as $I$. Furthermore, Lemma \ref{lemma-subspaces} implies that if $G$ is a subspace of $H^\perp$ of dimension $j$ (recall that $j\leq n-i=\textnormal{dim}H^\perp$), then for every $K\subseteq G$, $K\in {\cal K}^n$, we have $\overline{\diamondsuit}K\subseteq G$. Thus, if $\diamondsuit'$ is the restriction of $\overline{\diamondsuit}$ onto the family of convex subsets of $G$, then $\diamondsuit':{\cal K}^j\to{\cal K}^j_{\{o\}}$ is a monotonic, invariant of $H$-symmetric sets and $V_j$-preserving $o$-symmetrization, where we identify $G$ with $\mathbb{R}^j$. We finally conclude that in order to prove Theorem \ref{thm-v_j}, we may assume that $j=n$, $i=0$ and that for every segment $I$ through the origin, $\diamondsuit I$ is the $o$-symmetric line segment that is parallel to $I$ and has the same length as $I$.

Let $\{e_1,\dots,e_n\}$ be an orthonormal basis of $\mathbb{R}^n$ and $T\subseteq F:=\textnormal{span}\{e_1,e_2\}$ be an equilateral triangle whose barycenter is at the origin. Consider the convex body $$K:=\textnormal{conv}(T\cup [-e_3,e_3]\cup\dots\cup[-e_n,e_n]),$$where ``conv'' denotes the convex hull.
Notice that by Lemma \ref{lemma-subspaces}, we have $\diamondsuit T\subseteq F$. Then, by the monotonicity of $\diamondsuit$ and the fact that $[-e_i,e_i]$ is $H$-symmetric, $i=3,\dots,n$, we obtain
$$\diamondsuit K\supseteq \textnormal{conv}(\diamondsuit T\cup [-e_3,e_3]\cup\dots\cup[-e_n,e_n])=\textnormal{conv}(((\diamondsuit T)\cap F)\cup [-e_3,e_3]\cup\dots\cup[-e_n,e_n])=:K'.$$
On the other hand, it is clear that
$$V_n(K)=\frac{2^{n-1}}{ n!}V_2(T)$$
and
$$V_n(K')=\frac{2^{n-1}}{n!}V_2((\diamondsuit T)\cap F),$$
which gives (since $\diamondsuit$ is volume preserving) $V_2(T)=V_2((\diamondsuit T)\cap F)$. However, Lemma \ref{lemma-triangle} implies  $V_2(T)<V_2((\diamondsuit T)\cap F)$, which is a contradiction, and the proof is complete. $\square$
\section{Characterizations of Minkowski symmetrization}\label{s-V_1}
We start with the following definition.
\begin{definition}\label{def-segment-property}Let $H$ be a subspace in $\mathbb{R}^n$ and $j\in\{1,\dots,n-1\}$. A symmetrization $\diamondsuit:{\cal K}^n\to{\cal K}^n_H$ is called \textit{canonical} if it maps line segments, contained in a translate of $H^\perp$, to line segments. If $\diamondsuit\{x\}=\{x\}$, for all $x\in H$, then $\diamondsuit$ is called $H$-\textit{invariant}. A symmetrization $\diamondsuit:{\cal K}_n^n\to {\cal K}_{n,H}^n$ is called canonical (resp. $H$-invariant) if its natural extension is canonical (resp. $H$-invariant).
\end{definition}
\hspace*{1.5em}The main goal of this section is to establish the following:
\begin{theorem}\label{thm3}
Let $n\in \mathbb{N}$, $i\in\{0,1\dots,n-1\}$, $H\in G _{n,i}$ and $\diamondsuit:{\cal B}\to{\cal B}_H$ a monotonic symmetrization, which is invariant on $H$-symmetric sets and mean width preserving, where ${\cal B}={\cal K}^n$ or ${\cal K}_n^n$. If, in addition, $\diamondsuit$ is canonical, then $\diamondsuit$ is invariant under translations orthogonal to $H$ of $H$-symmetric sets.
\end{theorem}
First we will need the following result established in \cite{Bi-Ga-Gr} for $i>0$. Below, we will also need the case $i=0$, so we include it here.
\begin{theorem}\label{old-thm}
Let $n\in \mathbb{N}$, $i\in\{0,1,\dots,n-1\}$ and $H\in G _{n,i}$. Let ${\cal B}={\cal K}^n$ or ${\cal K}^n_n$. If $\diamondsuit:{\cal B}\to{\cal B}_H$ is a monotonic, mean width preserving symmetrization, invariant on $H$-symmetric sets and invariant under translations orthogonal to $H$ of $H$-symmetric sets, then $\diamondsuit=M_H$.
\end{theorem}
\begin{proof}
As mentioned earlier, the case $i>0$ is already settled in \cite{Bi-Ga-Gr}, thus we only have to prove our claim for $i=0$. Let $K\in{\cal B}$, $u\in \mathbb{S}^{n-1}$. Then, there exists a box $P$ that contains $K$, one of its facets is contained in the supporting hyperplane of $K$ whose unit normal vector is $u$ and one of its facets is contained in the supporting hyperplane of $K$ whose unit normal vector is $-u$. Then it is clear that $h_K(\pm u)=h_P(\pm u)$. Since $P$ is a translate of an $o$-symmetric set, we have:
$$h_{\diamondsuit K}(u)\leq h_{\diamondsuit P}(u)=\frac{h_P(u)+h_P(-u)}{2}=\frac{h_K(u)+h_K(-u)}{2}.$$
The latter inequality is true for all $u\in \mathbb{S}^{n-1}$, thus
$$\diamondsuit K\subseteq \frac{1}{2}K+\frac{1}{2}(-K)=M_{\{o\}}K$$and since $\diamondsuit$ is $V_1$-preserving, we conclude that $\diamondsuit K=M_{\{o\}}K$, as asserted.
\end{proof}

Combining Theorems \ref{thm3} and \ref{old-thm}, we immediately obtain:
\begin{corollary}\label{cor-mink-sym-1}
Let $n\in \mathbb{N}$, $i\in\{0,1\dots,n-1\}$, $H\in G _{n,i}$ and $\diamondsuit:{\cal B}\to{\cal B}_H$ a monotonic symmetrization, which is invariant on $H$-symmetric sets and mean width preserving, where ${\cal B}={\cal K}^n$ or ${\cal K}_n^n$. If, in addition, $\diamondsuit$ is canonical, then $\diamondsuit=M_H$.
\end{corollary}

For the proof of Theorem \ref{thm3}, we will need some geometric lemmas.
The next lemma follows immediately from the definitions and its proof is omitted.
\begin{lemma}\label{lemma-mink-sym-no-proof}
Let $H$ be a proper subspace of $\mathbb{R}^n$. The following statements hold:
\begin{enumerate}[i)]
\item A line segment $I$ is $H$-symmetric if and only if for every straight line $\varepsilon\subseteq H$, such that $\varepsilon\cap I\neq \emptyset$, $I$ is $\varepsilon$-symmetric.
\item Let $K\in{\cal K}^n$ be an $H$-symmetric set. Then,
$$K=\bigcup\{I:I\textnormal{ is an }H\textnormal{-symmetric segment contained in }K\}.$$
 \end{enumerate}
\end{lemma}
\begin{lemma}\label{lemma-reduction-1}
Let $n\in \mathbb{N}$, $i\in\{0,1\dots,n-1\}$, $H\in G _{n,i}$ and $\diamondsuit:{\cal K}^n\to{\cal K}^n_H$ a monotonic $V_j$-preserving map, for some $j\in\{1,\dots,n\}$. If for any symmetric line segment $J$ and for any $x\in H^{\perp}$, it holds $\diamondsuit(J+x)=J$, then $\diamondsuit$ is invariant under translations orthogonal to $H$ of $H$-symmetric sets.
\end{lemma}
\begin{proof}
Let $K\in{\cal K}^n$ be an $H$-symmetric set and $x\in H^{\perp}$. Note that $K+x$ is $(H+x)$-symmetric. Thus, by Lemma \ref{lemma-mink-sym-no-proof} $(ii)$, we have:
$$K+x=\bigcup\{I+x:I\textnormal{ is an }H\textnormal{-symmetric segment contained in }K\}.$$
Thus, by the fact that $\diamondsuit(I+x)=I$, for any $H$-symmetric segment $I$ and by the monotonicity of $\diamondsuit$, we immediately get: $\diamondsuit(K+x)\supseteq K$. The assertion follows by the fact that $V_j(\diamondsuit(K+x))=V_j(K+x)=V_j(K)$.
\end{proof}

Next, we would like to show how the proof of Theorem \ref{thm3} reduces to a 2-dimensional problem.
\begin{lemma}\label{lemma-reduction-2}
Theorem \ref{thm3} reduces to the following: Let $\diamondsuit:{\cal K}^2\to {\cal K}^2_{\{o\}}$ be a monotonic symmetrization, which is invariant on $o$-symmetric sets, $V_1$-preserving and canonical. Then, for any $o$-symmetric line segment $J$ and for any $x\in\mathbb{R}^2$, it holds $\diamondsuit (J+x)=J$.
\end{lemma}
\begin{proof}
First assume that we have proved Theorem \ref{thm3} for ${\cal B}={\cal K}^n$. Let $\diamondsuit:{\cal K}^n_n\to{\cal K}^n_n$ be a symmetrization that satisfies the conditions of Theorem \ref{thm3}. Then, Lemma \ref{lemma-natural-extension} guarantees that its natural extension $\overline{\diamondsuit}:{\cal K}^n\to{\cal K}_H^n$ also satisfies the conditions of Theorem \ref{thm3}. Consequently (since we assumed that Theorem \ref{thm3} is established for ${\cal B}={\cal K}^n$), $\overline{\diamondsuit}:{\cal K}^n\to{\cal K}_H^n$ is invariant under translations orthogonal to $H$ of $H$-symmetric sets. Hence, again by Lemma \ref{lemma-natural-extension}, $\diamondsuit$ is invariant under translations orthogonal to $H$ of $H$-symmetric sets. Therefore, we may assume that ${\cal B}={\cal K}^n$.

Note that since $V_1$ is clearly strictly increasing, it follows by Theorem \ref{thm2} that $\diamondsuit$ is projection invariant. Also, by Lemma \ref{lemma-reduction-1}, it suffices to show that for any $H$-symmetric line segment $J$ and for any $x\in H^{\perp}$, $\diamondsuit(J+x)=J$.
In other words,
we need to show that for any 2-dimensional affine subspace $F$ 
of $\mathbb{R}^n$, which is perpendicular to $H$, 
for any $(H\cap F)$-symmetric line segment $J\subseteq F$
and for any $x\in F$, we have: $\diamondsuit (J+x)=J$. Since $F$ can be naturally identified with $\mathbb{R}^2$ 
(in that case, $H\cap F=\{o\}$) and since the restriction of $\diamondsuit$ onto the family of compact convex subsets of $F$ satisfies the conditions of Theorem \ref{thm3}, the proof of our lemma is complete.

\end{proof}
\begin{lemma}\label{lemma-mink-sym-plane-geom}
Let $\varepsilon,\varepsilon'$ be two non-parallel straight lines in $\mathbb{R}^2$. Assume that $\varepsilon'$ contains the origin. Then, there exists a line segment $I\subseteq \varepsilon$, an $o$-symmetric line segment $I'\subseteq \varepsilon'$ of the same length as $I$ and an $o$-symmetric convex body $K\subseteq \mathbb{R}^2$ that contains $I$ but does not contain $I'$.
\end{lemma}
\begin{proof}
Fix two non-parallel lines $\varepsilon,\varepsilon'$ in $\mathbb{R}^2$, such that $\varepsilon'\ni o$. Set $a$ to be the distance of $\varepsilon$ from $o$. Denote by $I'$ the $o$-symmetric line segment of length 1 that is contained in $\varepsilon'$. If $\varepsilon$ contains $o$, then we can take $K$ to be the convex hull of the $o$-symmetric line segment of length 1, contained in $\varepsilon$ and the line segment $\frac{1}{2}I'$. Then, $K$ is $o$-symmetric, contains $I$ and does not contain $I'$, as required. Therefore, we may assume that the origin is not contained in $\varepsilon$.

Denote by $\zeta$ the line through $o$, which is orthogonal to $\varepsilon'$. Also, consider the point $A:=\varepsilon\cap \zeta$ and the line segment $I\subseteq \varepsilon$, which is centered at $A$ and has length 1. Set $K$ to be the (unique) $o$-symmetric parallelogram that has $I$ as one of its sides. Let $J$ be one of the sides of $K$ which are not parallel to $I$. Then, $-J$ is a side of $K$, parallel to $J$. Clearly, the area of $K$ equals $2a$. Assume that $K\supseteq I'$. Then, the distance of $J$ and $-J$ is greater than or equal to the length of $I'$, i.e. 1. However, since $\varepsilon,\varepsilon' $ are not parallel, $J$ cannot be perpendicular to $I$. Thus, the triangle inequality implies that the length of $J$ is strictly greater than the distance of $I,-I$, i.e. $2a$. This shows that the area of $K$ is strictly greater than $2a$. We arrived at a contradiction because we assumed that $K$ contains $I'$. Since $K$ trivially contains $I$, our proof is complete.
\end{proof}
\textbf{}\\
Proof of Theorem \ref{thm3}.\\
Let $\diamondsuit:{\cal K}^2\to {\cal K}^2_{\{o\}}$ be a monotonic, mean width preserving symmetrization, which maps line segments to line segments. Fix an $o$-symmetric line segment $J$ and a vector $x\in\mathbb{R}^2$. As Lemma \ref{lemma-reduction-2} shows, to prove Theorem \ref{thm3}, it suffices to prove that $\diamondsuit(J+x)=J$. First note that since $\diamondsuit$ is $V_1$-preserving, the line segments $J$ and $\diamondsuit(J+x)$ have the same length. Assume that $\diamondsuit(J+x)\neq J$.
Then, the segments $\diamondsuit(J+x)$ and $J$ cannot be parallel (otherwise, since they are both $o$-symmetric and of the same length, they would coincide). Set $\varepsilon$ to be the straight line that contains $J+x$ and $\varepsilon'$ to be the straight line that contains $J$. Take any line segment $I\subseteq \varepsilon$. We claim that $\diamondsuit I\subseteq \varepsilon'$. Indeed, $I$ is contained in some line segment $I_0\subseteq \varepsilon$, that contains $J+x$. Consequently,
$\diamondsuit I_0\supseteq \diamondsuit  (J+x)$. Since $\diamondsuit I_0$ is a line segment, $\diamondsuit I_0$ has to be contained in $\varepsilon'$. But then, by monotonicity, $\diamondsuit I\subseteq\diamondsuit I_0\subseteq \varepsilon'$, as claimed.
Thus, we have shown that for any line segment $I\subseteq \varepsilon$, $\diamondsuit I$ is a centered line segment contained in $\varepsilon'$ and has the same length as $I$. However, since the lines $\varepsilon,\varepsilon'$ are not parallel, Lemma \ref{lemma-mink-sym-plane-geom} clearly implies that we can find a line segment $I\subseteq \varepsilon$ and an $o$-symmetric convex body $K$, such that $K$ contains $I$ but does not contain $\diamondsuit I$. Since $\diamondsuit K=K$, this contradicts the monotonicity of $\diamondsuit$. It follows that $\diamondsuit(J+x)=\diamondsuit J$, as required. $\square$\\
\\

The proof Lemma \ref{lemma-reduction-2} shows that if the answer to the following problem is affirmative, then we can remove the simplicity assumption in Theorem \ref{thm3}.
\begin{problem}
Let $\diamondsuit:{\cal K}^2\to {\cal K}^2_{\{o\}}$ be a symmetrization which is monotonic, invariant on $o$-symmetric sets and $V_1$-preserving. Is it true that $\diamondsuit$ is canonical (i.e. maps segments to segments)?
\end{problem}

Before ending this note, we would like to slightly relax the conditions of Corollary \ref{cor-mink-sym-1} as follows:
\begin{corollary}\label{cor-mink-last}
Let $n\in \mathbb{N}$, $i\in\{0,1\dots,n-1\}$, $H\in G _{n,i}$ and $\diamondsuit:{\cal B}\to{\cal B}_H$ a symmetrization which is monotonic, mean width preserving, canonical and $H$-invariant, where ${\cal B}={\cal K}^n$ or ${\cal K}_n^n$. If, in addition, the restriction of 
$\diamondsuit$ onto the family of sets in ${\cal B}$ that lie in $H^{\perp}$
is invariant on $o$-symmetric sets, then $\diamondsuit=M_H$.
\end{corollary}
\begin{proof}
As always, we may assume that ${\cal B}={\cal K}^n$. By Corollary \ref{cor-mink-sym-1}, we need to show that $\diamondsuit$ is invariant on $H$-symmetric sets. Since by Lemma \ref{lemma-mink-sym-no-proof}, any $H$-symmetric convex set $K$ can be written as:
$$K=\bigcup\{I:I\subseteq K\textnormal{ is an }H\textnormal{-symmetric line segment}\},$$
it is actually enough to show that $\diamondsuit$ is invariant on $H$-symmetric segments. Indeed, by the monotonicity of $\diamondsuit$, we would have: $\diamondsuit K\supseteq K$ and by the fact that $\diamondsuit$ is $V_1$-preserving, we would actually have  $\diamondsuit K=K$.

Let $I$ be an $H$-symmetric segment. Denote by $\varepsilon\subseteq H$ the line through the point $y:=H\cap I$ and the origin. Fix a point $x\in H$, such that $x$ is contained in the half-line defined by $o$ and $y$ but $x$ is not contained in the line segment $[o,y]$. Next, choose an origin symmetric line segment $J$, that is contained in the plane spanned by $\varepsilon$ and $I$, it is perpendicular to $\varepsilon$ and its length is so large that $I$ is contained in $P:=\textnormal{conv}(J\cup \{x\})$. Since $J\subseteq H^{\perp}$ and $J$ is $o$-symmetric, it follows that $J$ is $H$-symmetric, thus $\diamondsuit J=J$. Moreover, $x\in H$, thus $\diamondsuit \{x\}=\{x\}$. It follows that $\diamondsuit P\supseteq \textnormal{conv}(J\cup\{x\})=P$ and since $\diamondsuit$ is $V_1$-preserving, we get $\diamondsuit P=P$. Thus, by the monotonicity of $\diamondsuit$, we have $\diamondsuit I\subseteq \diamondsuit P=P$. However, since $\diamondsuit I$ is an $H$-symmetric line segment, it follows that $\diamondsuit I$ is perpendicular to $\varepsilon$, thus $\diamondsuit I$ is parallel to $I$ (here we used the fact that both $\diamondsuit I$ and $I$ are contained in the same plane, i.e. the plane that contains $P$). On the other hand, $\diamondsuit I\supseteq \diamondsuit\{y\}=\{y\}$, thus $\diamondsuit I$ is centered at $y$. Since $I$ and $\diamondsuit I$ have the same length, they must coincide and our claim follows.
\end{proof}
\textit{}\\
{\bf Acknowledgments.} I would like to thank Artem Zvavitch and the anonymous referee for many useful suggestions concerning the presentation of this note.

\vspace{2 cm}

\noindent Christos Saroglou \\
Department of Mathematical Sciences\\
Kent State University\\
Kent, OH 44242, USA \\
E-mail address: csaroglo@kent.edu \ \&\ christos.saroglou@gmail.com


\begin{thebibliography}{999}
\bibitem{girls} J. Abardia-Ev\'{e}quoz, and E. Saor\'{i}n G\`{o}mez, The role of the Rogers-Shephard inequality in the characterization of the difference body, Forum Mathematicum, 29 (2017), 1227-1243.
\bibitem{Bi-Ga-Gr} G. Bianchi, R. J. Gardner and P. Gronchi, Symmetrization in Geometry, Adv. Math. vol. 306 (2017), 51-88.
\bibitem{Bu} H. Busemann, Volumes in terms of concurrent cross-sections, Pacific J. Math. vol. 3 (1953), 1-12.
\bibitem{G2} R. J. Gardner, The Brunn-Minkowski inequality, Bull. Amer. Math. Soc. (N.S.) vol. 39 (2002), 355-405.
\bibitem{G} R. J. Gardner, Geometric tomography. Second edition. Encyclopedia of Mathematics and its Applications, vol. 58. Cambridge University Press, New York, 2006. xxii+492 pp. ISBN: 0-521; 0-521-68493-5.
\bibitem{G-H-W-1} R. J. Gardner, D. Hug and W. Weil, Operations between sets in geometry, J. Eur. Math. Soc. (JEMS) vol. 15 (2013), 2297-2352.
\bibitem{G-H-W-2} R. J. Gardner, D. Hug and W. Weil, The Orlicz-Brunn-Minkowski theory: a general framework, additions, and inequalities, J. Differential Geom. vol. 97 (2014), 427-476.
\bibitem{H-S} C. Haberl and F. E. Schuster, Asymmetric affine Lp Sobolev inequalities, J. Func. Anal. vol. 257 (2009), 641-658.
\bibitem{Mo} M. Ludwig, Minkowski valuations, Trans. Amer. Math. Soc. vol. 357 (2005), 4191-4213.
\bibitem{LYZ}  E.  Lutwak,  D. Yang and  G.  Zhang, Volume inequalities for subspaces  of $L^p$,  J. Differential Geom. vol. 68 (2004), 159-184
\bibitem{LYZ2} E. Lutwak, D. Yang and G. Zhang, Orlicz centroid bodies, J. Differential Geom. vol. 84 (2010), 365-387.
\bibitem{Me-Pa} M. Meyer and A. Pajor, On the Blaschke-Santal\'{o} inequality, Arch. Math. vol. 55 (1990), 82-93.
\bibitem{S-W} F. Schuster and T. Wannerer, Even Minkowski valuations, Amer. J. Math. vol. 137 (2015), 1651-1683.
\bibitem{S} R. Schneider, Convex bodies: the Brunn-Minkowski theory. Second expanded edition. Encyclopedia of Mathematics and its Applications. vol. 151. Cambridge University Press, Cambridge, 2014. xxii+736 pp. ISBN: 978-1-107-60101-7
\bibitem{S-S}  R. Schneider and F. Schuster, Rotation equivariant Minkowski valuations, Int. Math. Res. Not. 2006, Art. ID 72894, 20 pp.


\end{thebibliography}
\end{document}